\documentclass[reqno]{amsart}

\usepackage{amsmath,amssymb,color}
\usepackage{amsfonts, amscd, epsfig, amsmath, amssymb,enumerate}
\usepackage{graphicx}
\usepackage{color}

\newtheorem{theorem}{Theorem}[section]
\newtheorem{lemma}[theorem]{Lemma}
\newtheorem{proposition}[theorem]{Proposition}

\newtheorem{remark}[theorem]{Remark}
\newtheorem{definition}{Definition}[section]

\numberwithin{equation}{section}

\newcommand{\h}{\eta}
\newcommand{\bbL}{\mathbb L}

\renewcommand{\epsilon}{\varepsilon}
\newcommand{\e}{\epsilon}

\newcommand{\R}{\mathbb R}

\newcommand{\Z}{\mathbb Z}
\newcommand{\N}{\mathbb N}
\renewcommand{\P}{\mathbb P}

\begin{document}
\title[]{Kinetically constrained lattice gases: tagged particle diffusion}
\author[O. Blondel ]{O. Blondel}
\email{blondel@math.univ-lyon1.fr}
\address{
Institut Camille Jordan CNRS-UMR 5208
B\^atiment Braconnier
Univ.\ Claude Bernard Lyon 1
43 boulevard du 11 novembre 1918
69622 Villeurbanne cedex}
\author[C. Toninelli]{C. Toninelli}
\email{cristina.toninelli@upmc.fr}
\address{Laboratoire de Probabilit\'es et Mod\`eles Al\'eatoires
  CNRS-UMR 7599 Univ.\ Paris VI-VII 4, Place Jussieu F-75252 Paris Cedex 05 France}
  \thanks{This work has been supported by the ERC Starting Grant 680275 MALIG
 }\begin{abstract}

 Kinetically constrained lattice gases (KCLG) are interacting particle systems on the integer lattice $\mathbb Z^d$ with   hard core
exclusion and 
 Kawasaki type dynamics. Their peculiarity  is that jumps are allowed only
if  the configuration satisfies a constraint 
which asks for enough empty sites in a certain local neighborhood.
KCLG have been introduced and extensively studied in physics literature as models of glassy dynamics.
We focus on the most studied class of KCLG, the Kob Andersen (KA) models. 
We analyze the behavior of a tracer (i.e.\@ a tagged particle) at equilibrium. We prove that for all dimensions $d\geq 2$
and for any equilibrium particle density, under diffusive rescaling the motion of the tracer converges to a $d$-dimensional Brownian motion with non-degenerate diffusion matrix. Therefore we disprove the occurrence of a diffusive/non diffusive transition  which had been conjectured  in physics literature. 
 Our technique is flexible enough and can be extended to analyse the tracer behavior for other choices of constraints.
\end{abstract}
\maketitle

{\sl MSC 2010 subject classifications}:{60K35}
{60J27}

{\sl Keywords}: Kawasaki dynamics, tagged particle, kinetically constrained models

\section{Introduction}

Kinetically constrained lattice gases (KCLG) are interacting particle systems on the integer lattice $\mathbb Z^d$ with hard core
exclusion, i.e.\@
with the constraint that on each site there is at most one particle.
A configuration is therefore defined by giving for each site
$x\in\mathbb Z^d$ the occupation variable $\eta(x)\in \{0,1\}$, which
represents an empty or occupied site respectively.
The dynamics is given by a continuous time Markov process of Kawasaki type, which allows 
the exchange of the occupation variables across a bond  $e=(x,y)$  of
neighboring sites $x$ and $y$ with a rate $c_{x,y}(\eta)$ depending on
the configuration $\h$.   
The simplest case is the simple symmetric  exclusion process
(SSEP) in which a jump of a particle to a neighboring empty site occurs at rate one, namely
 $c_{x,y}^{{SSEP}}(\h)=(1-\eta(x))\eta(y)+\eta(x)(1-\eta(y))$. 
Instead, for KCLG  the jump to a neighboring empty site can occur only if the configuration satisfies a certain local constraint 
which involves the occupation variables on other sites besides the initial and final position of the particle. More precisely $c_{x,y}(\eta)$ is of the form 
$c_{x,y}^{SSEP}r_{x,y}(\eta)$ where $r_{x,y}(\eta)$  degenerates to zero for certain choices of $\{\eta(z)\}_{z\in\mathbb Z^d\setminus \lbrace x,y\rbrace}$. 
Furthermore $r_{x,y}$ does 
not depend on the value of $\eta(x)$ and $\eta(y)$ and therefore
detailed balance w.r.t.\@ $\rho$-Bernoulli product
measure $\mu_{\rho}$ is verified for any $\rho\in[0,1]$. Therefore $\mu_{\rho}$ is an
invariant reversible measure for the process. However, at variance with 
the simple symmetric exclusion process, KCLG have several
other invariant measures. This is related to the fact that due to the degeneracy of $r_{x,y}(\eta)$
there exist \emph{blocked configurations}, namely configurations for which all
exchange rates are equal to zero.\\

\noindent KCLG have been introduced in physics literature (see \cite{Ritort,GST} for a
review) to model the liquid/glass transition that occurs when a liquid is suddenly cooled.
In
particular they were devised to mimic the fact that the motion of a
molecule in a low temperature (dense) liquid can be inhibited by the geometrical
constraints created by the surrounding molecules.  Since the exchange rates are devised to
encode this local caging mechanism, they require 
a minimal number of empty sites
in a certain neighborhood of $e=(x,y)$ in order for the exchange at $e$
to be allowed. There exists also a non-conservative version of KCLG, the so called Kinetically Constrained Spin Models, which feature a Glauber type dynamics and have been recently studied in several works (see e.g.\@ \cite{CMRT,orianediff} and references therein).

Let us start by recalling some fundamental issues which, due to the fact that the jump to a neighboring empty site is not always allowed,  require for KCLG 
different techniques
from those used to study SSEP. 
A first basic question is whether the infinite volume process is
ergodic, namely whether zero is a simple eigenvalue for the
generator of the Markov process in $\bbL_2(\mu_{\rho})$. 
This would in turn imply relaxation to  
$\mu_{\rho}$ in the $\bbL_2(\mu_{\rho})$ sense. Since the constraints require a minimal number of empty sites,  it is possible that the process undergoes a
transition from an ergodic to a non ergodic regime at $\rho_c$ with
$0<\rho_c<1$. 
  The
next natural issue is to establish the large time behavior of the
infinite volume process in the  ergodic regime, when we start from equilibrium measure at time
zero. This in turn is related to  the scaling with the
system size of the spectral gap and of the inverse of the log Sobolev
constant on a finite volume. Recall that
for SSEP decay to equilibrium occurs as $1/t^{d/2}$ and both the 
spectral gap and the inverse of the log Sobolev constant decay as $1/L^2$ uniformly in
the density $\rho$ \cite{quastel,yau}, where $L$ is the linear size of the finite volume. Numerical simulations for some KCLG suggest the possibility of an anomalous slowing down at high density \cite{KA,MP} which could correspond to a scaling of the spectral gap and of the log Sobolev constant different from SSEP. 
Two other natural issues are the evolution of macroscopic density profiles, namely the
study of the hydrodynamic limit, and the large time behavior of a tracer particle under a diffusive rescaling. 
For SSEP and $d\geq 2$ the tracer particle converges to a Brownian motion \cite{spohn}, more precisely the rescaled position of the tracer at time $\epsilon^{-2} t$ converges as $\epsilon\to 0$, to a $d$-dimensional Brownian motion with non-degenerate diffusion matrix. Instead, for some KCLG it has been conjectured that a diffusive/non-diffusive transition  occurs at a finite critical density $\rho_c<1$:  the self-diffusion matrix would be non-degenerate only for $\rho<\rho_c$  \cite{KA,kurchan2}. Concerning the hydrodynamic limit, the following holds for SSEP:
starting from an initial condition that has
a density profile and under a diffusive rescaling, there is a density profile at later times  and it  can be obtained from the initial
one by solving  the heat
equation \cite{S}.  For KCLG a natural candidate for the
hydrodynamic limit is a parabolic equation of porous media type
degenerating when the density approaches one. Establishing this result in presence of constraints is particularly challenging. \\

In order to recall the previous results on KCLG and to explain the novelty of our results, we should distinguish among {\sl cooperative} and {\sl non-cooperative} KCLG.
 A
model is said to be non-cooperative if its constraints are such that it is
possible to construct a proper finite group of vacancies, {\sl the
  mobile cluster}, with the following two properties: 
(i) for any configuration it is possible to move the mobile cluster to
  any other position in the lattice by a sequence of allowed exchanges;  (ii)
 any nearest neighbor exchange is allowed if the mobile cluster is in a proper position in its
vicinity.
All models which are not non-cooperative are said to be
cooperative. From the point of view of the modelization of the liquid/glass
transition, cooperative models are the most relevant ones.  Indeed, very
roughly speaking, non cooperative models are expected to behave like a
rescaled SSEP with the mobile cluster playing the role of a single
vacancy and are less suitable to describe the rich
behavior of glassy dynamics. Furthermore, from a mathematical point of view, cooperative models are much more challenging. Indeed, for non-cooperative models the existence of finite mobile clusters simplifies the analysis and allows the application of some standard techniques (e.g.\@ paths arguments) already developed for SSEP. \\

We can now recall 
the existing mathematical results for KCLG.\\
\noindent
{\sl Non-cooperative models.}
Ergodicity in infinite volume at any $\rho<1$ easily follows from the fact that with probability one there exists a mobile cluster and using path arguments (see for example \cite{bertini-toninelli}).
 In \cite{bertini-toninelli} it is proven in certain cases that both the inverse of the spectral gap and the log Sobolev
constant in finite volume of linear size $L$ with boundary sources \footnote{Namely with the addition of Glauber birth/death terms at the boundary} scale as $O(L^2)$. Furthermore for the same models
the self-diffusion matrix of the tagged particle
is proved to be non-degenerate \cite{bertini-toninelli}. The diffusive scaling of the spectral gap has been proved also for some models without boundary sources in \cite{nagahata}.  Finally, the hydrodynamic limit has been successfully analyzed for a special class constraints in \cite{GLT}. In all these cases the 
macroscopic density evolves under diffusive rescaling according to a
porous medium equation of the type $\partial_t\rho(t,u)=\nabla( D \nabla 
\rho)$ with $D(\rho)=(1-\rho)^m$ and $m$ an integer parameter. \vspace{0.15 cm}

\noindent
\emph{Cooperative models.}
The class of cooperative models  which has been most studied in physics literature are the so-called {\sl Kob
Andersen (KA) models}  \cite{KA}. 
KA actually denotes a class of models on
$\mathbb Z^d$ characterized by an integer parameter $s$ with $s\in[2,d]$. The  nearest neighbor exchange rates  are defined as follows: $c_{x,y}=c_{x,y}^{SSEP}r_{x,y}(\eta)$ 
with $r_{x,y}=1$ if at least $s-1$ neighbors of $x$ different from
$y$ are empty and at least $s-1$ neighbors of $y$ different from $x$
are empty too, $r_{x,y}=0$ otherwise. In other words, a particle is allowed to jump to a neighboring empty site iff it has at least $s$ empty neighbors both in its initial and final position. 
Hence $s$ is called the {\sl facilitation parameter}. The choices $s=1$ and $s>d$ are discarded for the following reasons: $s=1$ coincides with SSEP, while for $s>d$ at any density the model is not ergodic \footnote{This follows from the fact that if $s>d$  there exists finite clusters of particles which are blocked. For example for $s=3,d=2$ if there is a  $2\times 2$ square fully occupied by particles all these particles can never jump to their neighboring empty position.}.
It is immediate to verify that KA is a cooperative model
for all  $s\in[2,d]$. For example if $s=d=2$ a fully
occupied double column which spans the lattice can never be
destroyed. Thus no finite cluster of vacancies can be mobile since
it cannot overcome the double column. 
 In \cite{TBF} it has
been proven that for all $s\in[2,d]$ the infinite volume process
is ergodic at any finite density, namely $\rho_c=1$, thus disproving previous conjectures \cite{KA,kurchan2,parisi} on the occurrence of an ergodicity breaking transition.
In \cite{CMRT2} a technique has been devised to analyze the spectral gap of
cooperative KCLG on finite volume with boundary sources. In particular, for KA model with $d=s=2$ it has been proved that in a box of linear size $L$ with
boundary sources, the spectral gap scales as $1/L^2$ (apart from logarithmic corrections)
 at any density. By using this result  it is proved that, again for the choice $d=s=2$, the infinite volume time auto-correlation of
local functions decays as $1/t$   (modulo logarithmic corrections) \cite{CMRT2}.
 The technique of \cite{CMRT2} can be extended to prove for all choices of $d$ and $s\in[2,d]$ a diffusive scaling for the spectral gap and a decay of the correlation at least as $1/t$. A lower bound as $1/t^{d/2}$ follows by comparison with SSEP.

\noindent In the present paper we 
analyze the behavior of a tracer (also called tagged particle) for KA models at equilibrium, namely when the infinite volume system is initialized with $\rho$-Bernoulli measure. We prove (Theorem \ref{mainth}) that  for all $d$, for any choice of $s\in[2,d]$ and for any $\rho<1$, under diffusive scaling the motion of the tracer converges to a $d$-dimensional Brownian motion with non-degenerate diffusion matrix.
Our result disproves the occurrence of a diffusive/non diffusive transition which had been conjectured  in physics literature on the basis of numerical simulations \cite{KA,kurchan2}.  Positivity of the  self-diffusion matrix at  any $\rho<1$ had been later claimed in \cite{BiTo}. However, the results in \cite{BiTo} do not provide a full and rigorous proof of the  positivity of the  self-diffusion matrix. Indeed, they rely on a comparison with the behavior of certain random walks in a random environment which is not exact. 
We follow here a novel route, different from the heuristic arguments sketched in \cite{BiTo}, which allows us to obtain the first rigorous proof of positivity of the self diffusion matrix for a cooperative KCLG. In particular we prove that positivity holds for any $\rho<1$ for all KA models. Our 
technique is flexible enough and can be extended to analyze other cooperative models in the ergodic regime.\\

\noindent The plan of the paper follows.
In Section \ref{model}, after setting  the relevant notation, we introduce KA models and state our main result (Theorem \ref{mainth}).
In Section \ref{ergodicity} we recall some basic properties of KA models: ergodicity  at any $\rho<1$ (Proposition \ref{ergodicity}); the existence of a finite critical scale above which with large probability a configuration  on finite volume can be connected to a {\sl framed} configuration, namely a configuration with empty boundary.
In Section \ref{sec:aux} we 
introduce  an auxiliary diffusion process which corresponds to a random walk on the infinite component of a certain percolation cluster. Then we prove that this auxiliary process has a non-degenerate diffusion matrix  (Proposition \ref{prop:auxiliary}). In Section \ref{comparison} we prove via path arguments  that the diffusion matrix of KA is lower bounded by the one for the  auxiliary process (Theorem \ref{Theo:compa}). This allows to conclude that the self diffusion matrix for KA model is non-degenerate.

\section{Model and results}\label{model}

The models considered here are defined on the integer lattice $\mathbb Z^d$
with sites $x = (x_1,\dots,x_d)$ and basis vectors $
e_1=(1,\dots,0)$, $ e_2=(0,1,\dots,0),\dots$, $
e_d=(0,\dots,1)$.  Given  $x$ and $y$ in $\mathbb Z^d$ we write $x\sim y$ if they are nearest neighbors, namely 
$d(x,y)=1$ where $d(\cdot,\cdot)$ is the distance associated with the Euclidean norm.
Also, given a finite set $\Lambda\subset \mathbb Z^d$ we define its neighborhood $\partial\Lambda$ as the set of sites outside $\Lambda$ at distance one and its interior neighborhood $\partial_-\Lambda$ as the set of sites inside $\Lambda$ at distance one from $\Lambda^c$, namely
$$\partial\Lambda:=\{x\not\in\Lambda:\exists y\in \Lambda {\mbox{ s.t. }} d(x,y)=1\}$$
$$\partial_-\Lambda:=\{x\in\Lambda:\exists y\in \Lambda^c {\mbox{ s.t. }} d(x,y)=1\}$$
We denote by $\Omega$ the configuration space, $\Omega=\{0,1\}^{\mathbb Z^d}$ and by the greek letters $\eta,\xi$ the configurations. Given $\eta\in\Omega$ we let $\eta(x)\in\{0,1\}$  be the occupation variable at site $x$.
We fix a parameter $\rho\in[0,1]$ and we denote by $\mu$ the $\rho$-Bernoulli product measure.
Finally, given $\eta\in\Omega$ for any bond $e=(x,y)$ we denote by $\eta^{xy}$ the configuration obtained from $\eta$ by exchanging the occupation variables at $x$ and $y$, namely 
$$
\eta^{xy} (z): = \left\{
\begin{array} {ll}
\eta(z) & \mbox{if } z \notin \{x,y\} \\
\eta(x) & \mbox{if } z=y\\
\eta(y) & \mbox{if } z=x.
\end{array}
\right.
$$
The Kob-Andersen (KA) models are  interacting particle systems with Kawasaki type (i.e.\@ conservative) dynamics on the lattice $\Z^d$ depending on a parameter $s\leq d$ (the facilitation parameter) with $s\in[2,d]$. They are Markov processes  defined through the generator which acts on local functions $f:\Omega\to \mathbb R$ as
\begin{equation}\label{Lenv}
L_{\rm env}f(\xi)=\sum_{x\in\Z^d}\sum_{y\sim x} c_{xy}(\xi)[f(\xi^{xy})-f(\xi)],
\end{equation}
where 
\begin{equation}\label{constraint}
c_{xy}(\xi)=\left\{\begin{array}{ll}
 1 & \text{if }\xi(x)=1,\ \xi(y)=0,\ \displaystyle\sum_{z\sim y}(1-\xi(z))\geq s-1\text{ and }\sum_{z\sim x}(1-\xi(z))\geq s,\\
0&\text{else}. 
\end{array}\right.
\end{equation}
where here and in the following with a slight abuse of notation  we let $\sum_{z\sim y}$ be the sum over sites $z\in \mathbb Z^d$ with $z\sim y$.
In words, each couple of neighboring sites $(x,y)$  waits an independent mean one exponential time and then the values $\eta(x)$ and $\eta(y)$ are exchanged provided : either (i) there is a particle at $x$ and an empty site at $y$ and at least $s-1$ nearest neighbors for $y$ and at least $s$ nearest neighbors for $x$ or (ii) there is a particle at $y$ and an empty site at $x$ and at least $s$ nearest neighbors for $y$ and at least $s-1$ nearest neighbors for $x$.
We call the jump of a particle from $x$ to $y$ \emph{allowed} if $c_{xy}(\xi)=1$. For any $\rho\in (0,1)$, the process is reversible w.r.t.\@ $\mu$, the product Bernoulli measure of parameter $\rho$.

We consider a tagged particle in a KA system at equilibrium. More precisely, we consider the joint process $(X_t,\xi_t)_{t\geq 0}$ on $\Z^d\times\lbrace 0,1\rbrace^{\Z^d}$ with generator \begin{eqnarray}
\mathcal{L}f(X,\xi)&=&\sum_{y\in\Z^d\setminus\lbrace X\rbrace}\  \sum_{z\sim y} c_{yz}(\xi)[f(X,\xi^{yz})-f(X,\xi)]\\
&&+\,\sum_{y\sim X}c_{Xy}(\xi)[f(y,\xi^{Xy})-f(X,\xi)]
\end{eqnarray}
and initial distribution $\xi_0\sim\mu_0:=\mu(\cdot|\xi(0)=1), X_0=0$. Here and in the rest of the paper, we denote for simplicity by $0$ the origin, namely site $x\in\mathbb Z^d$ with $e_i\cdot x=0$  $\forall i\in\lbrace 1,\dots,d\rbrace$.

In order to study the position of the tagged particle, $(X_t)_{t\geq 0}$, it is convenient to define the process of the environment seen from the tagged particle $(\eta_t)_{t\geq 0}:=(\tau_{X_t}\xi_t)_{t\geq 0}$, where $(\tau_x\xi)(y)=\xi(x+y)$. This process is Markovian, has generator
\begin{eqnarray}\label{L}
Lf(\eta)=\sum_{y\in\Z^d\setminus\lbrace 0\rbrace}\  \sum_{z\sim y} c_{yz}(\eta)[f(\eta^{xy})-f(\eta)]+\sum_{y\sim 0}c_{0y}(\eta)[f(\tau_{y}(\eta^{0y}))-f(\eta)]
\end{eqnarray}
and is reversible w.r.t.\@ $\mu_0$. We still say that the jump of a particle from $x$ to $y$ is an \emph{allowed move} if $c_{xy}(\eta)=1$. In the case $x=0$, this jump in fact turns $\eta$ into $\tau_{y}(\eta^{0y})$.
By using the fact that the process seen from the tagged particle is ergodic at any $\rho<1$  (see Proposition \ref{ergoprop}) we can apply a classic result \cite{spohn} and obtain the following.
\begin{proposition}\cite{spohn}\footnote{This result is proved in \cite{spohn} for exclusion processes on $\Z^d$ but the proof also works in our setting.}
\label{varD}
For any $\rho\in(0,1)$, there exists a non-negative $d\times d$ matrix $D(\rho)$ such that
\begin{equation}
\e X_{\e^{-2}t}\underset{\e\rightarrow 0}{\longrightarrow}\sqrt{2D(\rho)}B_t,
\end{equation}
where $B$ is a standard $d$-dimensional Brownian motion and the convergence holds in the sense of
weak convergence of path measures on $D([0,\infty),\R^d)$. Moreover, the matrix $D(\rho)$ is characterized by
\begin{multline}\label{varformula}
u\cdot D(\rho)u=\inf_{f}\left\{\sum_{y\in\Z^d\setminus\lbrace 0\rbrace}\  \sum_{z\sim y} \mu_0\left(c_{yz}(\eta)[f(\eta^{xy})-f(\eta)]^2\right)\right.\\
\qquad+\left.\sum_{y\sim 0}\mu_0\left(c_{0y}(\eta)[u\cdot y+f(\tau_{y}(\eta^{0y}))-f(\eta)]^2\right)\right\}
\end{multline}
for any $u\in\R^d$, where the infimum is taken over local functions $f$ on $\lbrace 0,1\rbrace^{\Z^d}$.
\end{proposition}

Our main result is the following.
\begin{theorem}
\label{mainth}Fix an integer $d$ and $s\in[2,d]$ and consider the KA model on $\mathbb Z^d$ with facilitation parameter $s$. Then, for any $\rho\in(0,1)$, any $i=1,\ldots,d$, we have $e_i\cdot D(\rho)e_i>0$.
In other words, the matrix $D(\rho)$ is non-degenerate at any density.
\end{theorem}

\begin{remark}
Since the constraints are monotone in $s$ (the facilitation parameter), it is enough to prove the above result for $s=d$. \emph{From now on we assume $s=d$.}
\end{remark}

\section{Ergodicity, frameability and characteristic lengthscale}
\label{ergodicity}

In this section we recall some key results for KA dynamics.
In \cite{CMRT2}, following the arguments of \cite{TBF}, it was proved that KA models are ergodic for any $\rho<1$. More precisely we have the following.
\begin{proposition}[Theorem 3.5 of \cite{CMRT2}]
Fix an integer $d$ and $s\in[2,d]$ and consider the KA model on $\mathbb Z^d$ with facilitation parameter $s$.
Fix $\rho\in(0,1)$ and let $\mu$ be the $\rho$-Bernoulli product measure. Then
$0$ is a simple eigenvalue of the generator $L_{\rm env}$ defined by formula \eqref{Lenv} considered on $L_2(\mu)$.
\end{proposition}
Along the same lines one can prove that the process of the environment seen from the tagged particle is ergodic on $L_2(\mu_0)$, namely recalling that $\mu_0:=\mu(\cdot|\xi(0)=1)$ it holds
\begin{proposition}\label{ergoprop}
$0$ is a simple eigenvalue of the generator $L$ defined by formula \eqref{L} considered on $L_2(\mu_0)$.
\end{proposition}

\begin{definition}[Allowed paths] \label{def:confpath}
Given $\Lambda\subset\mathbb Z^d$ and two configurations
$\eta,\sigma  \in \Omega$, a sequence of configurations
$$P_{\eta,\sigma}= (\eta^{(1)},\eta^{(2)},\ldots,\eta^{(n)})$$
starting at $\eta^{(1)}=\eta$ and ending at $\eta^{(n)}=\sigma$ is an
\emph{allowed  path}  from $\eta$ to $\sigma$ inside $\Lambda$ if for any
$i=1,\ldots,n-1$ there exists a bond $(x_i,y_i)$, namely a couple of neighboring sites, with
$\eta^{(i+1)}=(\eta^{(i)})^{x_iy_i}$ and $c_{x_iy_i}(\eta^{(i)})
=1$. We also require that paths do not go through the same configuration twice, namely for all $i,j\in[2,n]$ with $i\neq j$ it holds $\eta^{(i)}\neq\eta^{(j)}$.
We say that $n$ is the length of the path. Of course the notion of allowed path depends on the choice of the facilitation parameter $s$ which enters in the definition of $c_{xy}$. 
It is also useful to define allowed paths for the process seen from the tagged particle.
The paths are defined as before, with the only difference that for any
$i=1,\ldots,n-1$ there exists a bond $(x_i,y_i)$, namely a couple of neighboring sites, with
 $c_{x_iy_i}(\eta^{(i)})
=1$ and \begin{itemize}\item either $x_i=0$ and $\eta^{(i+1)}=\tau_{y_i}\left(\left(\eta^{(i)}\right)^{0y_i}\right)$  \item or  $x_i\neq 0$ and $\eta^{(i+1)}=(\eta^{(i)})^{x_iy_i}$ \end{itemize}
\end{definition}

Following the terminology of \cite{TBF} we introduce the notion of {\sl frameable} and {\sl framed} configurations.
\begin{definition}[Framed  and frameable configurations]
Fix a set $\Lambda \subset \mathbb Z^d$ and a configuration $\omega \in \Omega$.
We say that $\omega$ is \emph{$\Lambda$-framed} 
 if $\omega (x) =0$ for any $x \in \partial_{-} \Lambda$.
Let
$\omega^{(\Lambda)}$ be the configuration equal to $\omega_{\Lambda}$ inside $\Lambda$ and
equal to 1 outside $\Lambda$.
We say that  $\omega$ is  \emph{$\Lambda$-frameable}
if there exist a $\Lambda$-framed configuration $\sigma^{(\Lambda)}$ 
with at least one allowed configuration path $P_{\omega^{(\Lambda)}\to\sigma^{(\Lambda)}}$ inside $\Lambda$ (by definition any  framed configuration is also frameable). Sometimes, when from the context  it is clear to which geometric set $\Lambda$ we are referring, we will drop $\Lambda$ in the names and just say framed and frameable configurations. Of course the notion of frameable configurations depends on the choice of the facilitation parameter $s$.
\end{definition}

The following result, proved in \cite{TBF,CMRT2}, shows that on a sufficiently large lengthscale frameable configurations are typical.

\begin{lemma}\cite[Lemma 3.4]{CMRT2}
\label{frameability}
For any dimension $d$, any $\rho<1$ and any $\e>0$, there exists $\Xi=\Xi(\rho,\e,d)<\infty$ such that, for the KA process in $\Z^d$ with facilitation parameter $d$, for $L\geq \Xi$ it holds
\begin{equation}
\mu\left(\xi \text{ is } \Lambda_L-\text{frameable}\right)\geq 1-\e
\end{equation}
where we set $\Lambda_L=[0,L]^d$.
\end{lemma}

\section{An auxiliary diffusion}\label{sec:aux}

In this section we will  introduce a bond percolation process on a properly renormalized lattice and an auxiliary diffusion which corresponds to a random walk on the infinite component of this percolation. Then we will prove that this auxiliary process has a non-degenerate diffusion matrix  (Proposition \ref{prop:auxiliary}). This result will be the key starting point of the next section, where we will prove our main Theorem \ref{mainth} by comparing the diffusion matrix of the KA model with the diffusion matrix of the auxiliary process (Theorem \ref{Theo:compa}).

In order to introduce our bond percolation process we need some auxiliary notation.
Fix a parameter $L\in\N$ and  consider the renormalized lattice $(L+2)\Z^d$.
For $n\in\lbrace 0,\ldots d\rbrace$, let $B^{(n)}:=\lbrace 0,1\rbrace^{d-n}\times\lbrace 0,\ldots L-1\rbrace^n$. We say that $B^{(n)}$ is the elementary block of $L$--dimension $n$. $B^{(n)}_1,\ldots B^{(n)}_{\binom{d}{n}}$ are the blocks obtained from $B^{(n)}$ by permutations of the coordinates. Notice that one can write the cube of side length $L+2$ as a disjoint union of such blocks in the following way (see Figure~\ref{fig:Lambdas}):
\begin{equation}\label{decomposition}
\Lambda_{L+2}:=\lbrace 0,\ldots,L+1\rbrace^d=B^{(0)}\sqcup\bigsqcup_{n=1}^d\bigsqcup_{i=1}^{\binom{d}{n}}\left(B^{(n)}_i+2e_{j_{i1}}+\ldots+2e_{j_{in}}\right),
\end{equation}
where the block $B^{(n)}_i$ has length $L$ in the directions $e_{j_{i1}},\ldots,e_{j_{in}}$ and length $2$ in the other directions. 
By first decomposing $\Z^d$ in blocks of linear size $L+2$ and then using this decomposition, we finally get a paving of $\Z^d$ by blocks with side lengths in $\lbrace 2, L\rbrace$.
\begin{figure}
\begin{center}
\includegraphics[width=.4\textwidth]{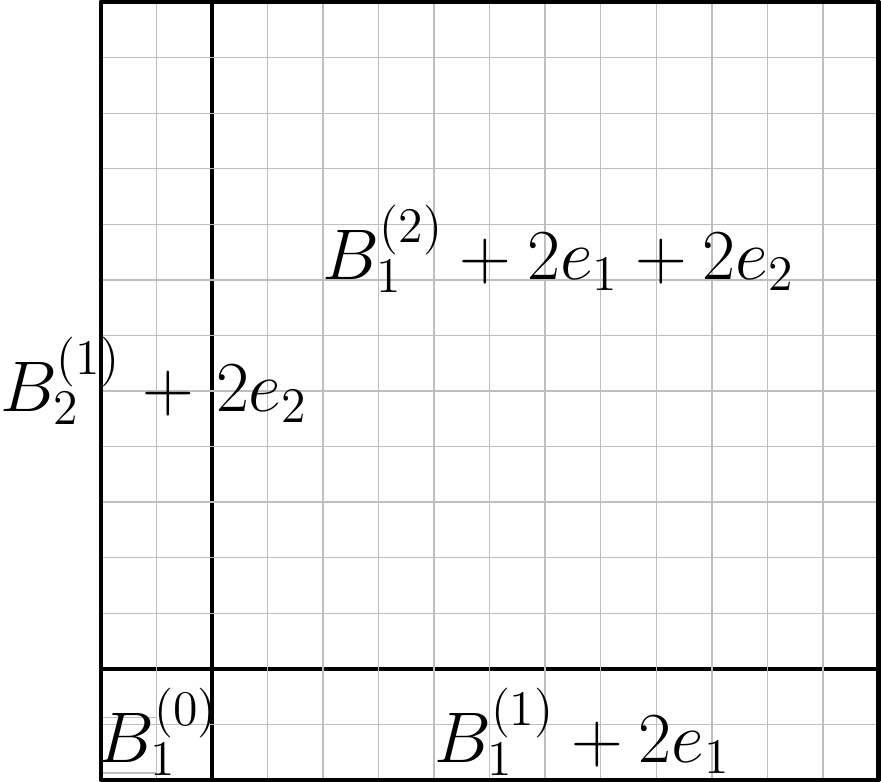}\hspace{.1\textwidth}\includegraphics[width=.4\textwidth]{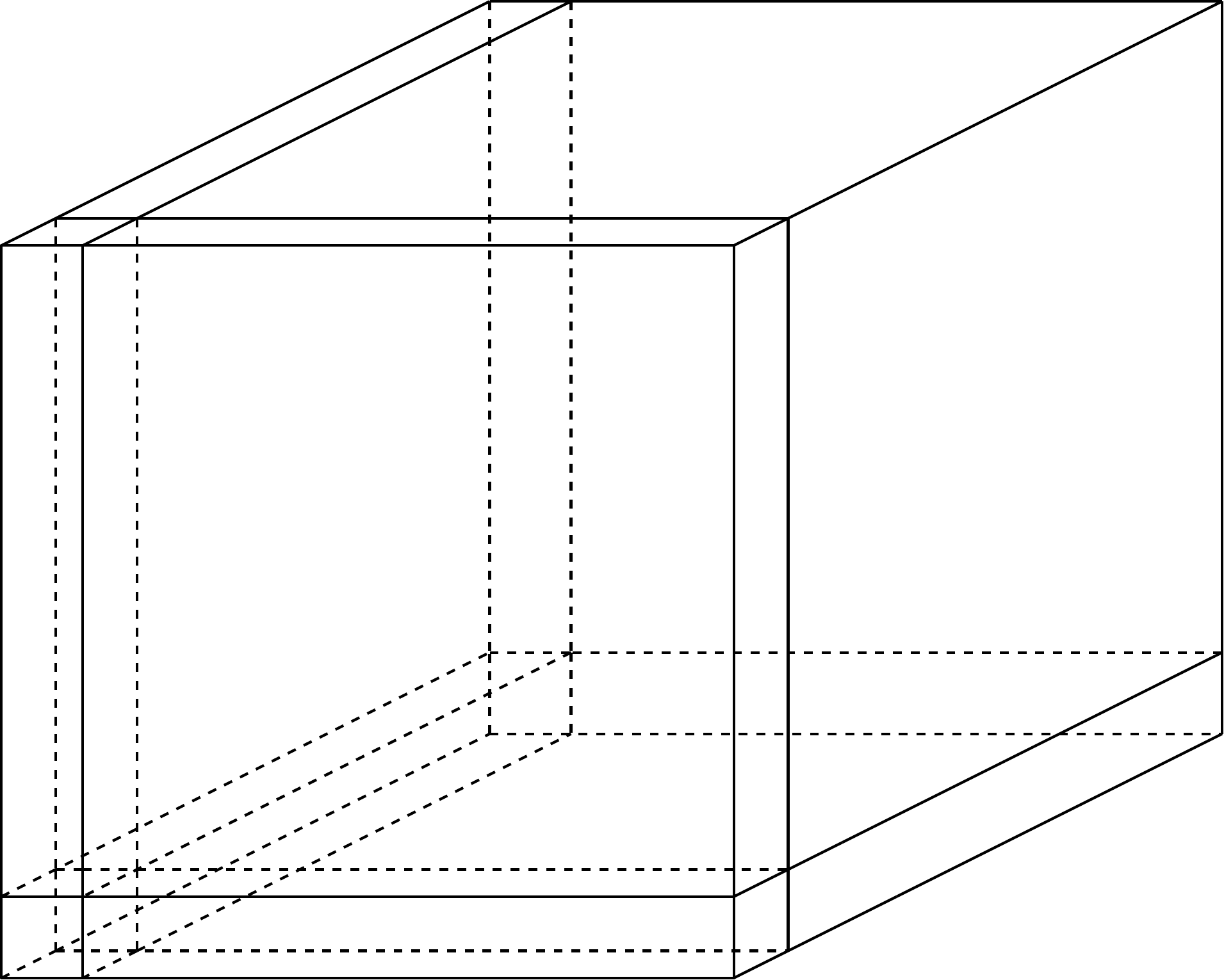}
\caption{The covering of $\lbrace 0,\ldots L+1\rbrace^d$ by blocks $B^{(n)}_i$ for $d=2,3$.}
\label{fig:Lambdas}
\end{center}
\end{figure}
We will speak of \emph{liaison tubes} or just \emph{tubes} for blocks of $L$--dimension $1$ and of \emph{facilitating blocks} for blocks of $L$--dimension $2$ or larger.
We will also call {\sl faces} of $B^{(n)}_i$ the $2^{d-n}$ (disjoint) regions of the form
$x_{j_{i1}}\in[0,L-1],\ldots,x_{j_{in}}\in[0,L-1]$ and $x_j=c_j$ with $c_j\in\{0,1\}$ for all $j\not\in\{j_{i1},\dots,j_{in}\}$.\\
Finally, for $x\in (L+2)\Z^d$, $i=1,\ldots,d$, we define the block neighborhood $\mathcal{N}_{x,i}$ of $(x,x+(L+2)e_i)$ recursively in the $L$--dimension of the blocks (see also Figure~\ref{fig:openedge}):
\begin{itemize}
\item $B^{(0)}+x$ and $B^{(0)}+x+(L+2)e_i$ belong to $\mathcal{N}_{x,i}$,
\item each tube adjacent to $B^{(0)}+x$ or $B^{(0)}+x+(L+2)e_i$ belongs to $\mathcal{N}_{x,i}$,
\item recursively, each block of $L$--dimension $n+1$ adjacent to some block of $L$--dimension $n$ in $\mathcal{N}_{x,i}$ is also in $\mathcal{N}_{x,i}$.
\end{itemize}

We are now ready to define our bond percolation process.
Let $\mathcal{E}((L+2)\Z^d)$ be the set of bonds of $(L+2)\Z^d$. Given a configuration $\eta\in\{0,1\}^{\mathbb Z^d}$,  the corresponding configuration on the bonds
$\bar \eta\in\{0,1\}^{\mathcal{E}((L+2)\Z^d)}$ is defined by $\bar\eta_{x,x+(L+2)e_i}=1$  iff \begin{enumerate}
\item \label{cond1} each tube in $\mathcal{N}_{x,i}$ contains at least a zero,
\item \label{cond2}
for all $n=2,\ldots,d$, for all $B$ block of $L$--dimension $n$ in $\mathcal{N}_{x,i}$, 
let $\Lambda_{B,i}$ with $i\in[1,2^{n-d}]$ be its faces. The configuration should be $\Lambda_{B,i}$ frameable
for  KA process  with parameter $n$.
\end{enumerate}
In other words the edge
 $(x,x+(L+2)e_i)$ is open if (1) and (2) are satisfied, closed otherwise. See Figure~\ref{fig:openedge} for an example of an open bond.

Note that conditions (1) and (2) do not ask anything of the configuration inside $B^{(0)}+x$ and $B^{(0)}+x+(L+2)e_i$. As a consequence, the distribution of $\bar\eta$ is the same for $\eta\sim\mu$ as for $\eta\sim\mu_0$. We denote it by $\bar\mu$.

\begin{figure}
\begin{center}
\includegraphics[scale=.4]{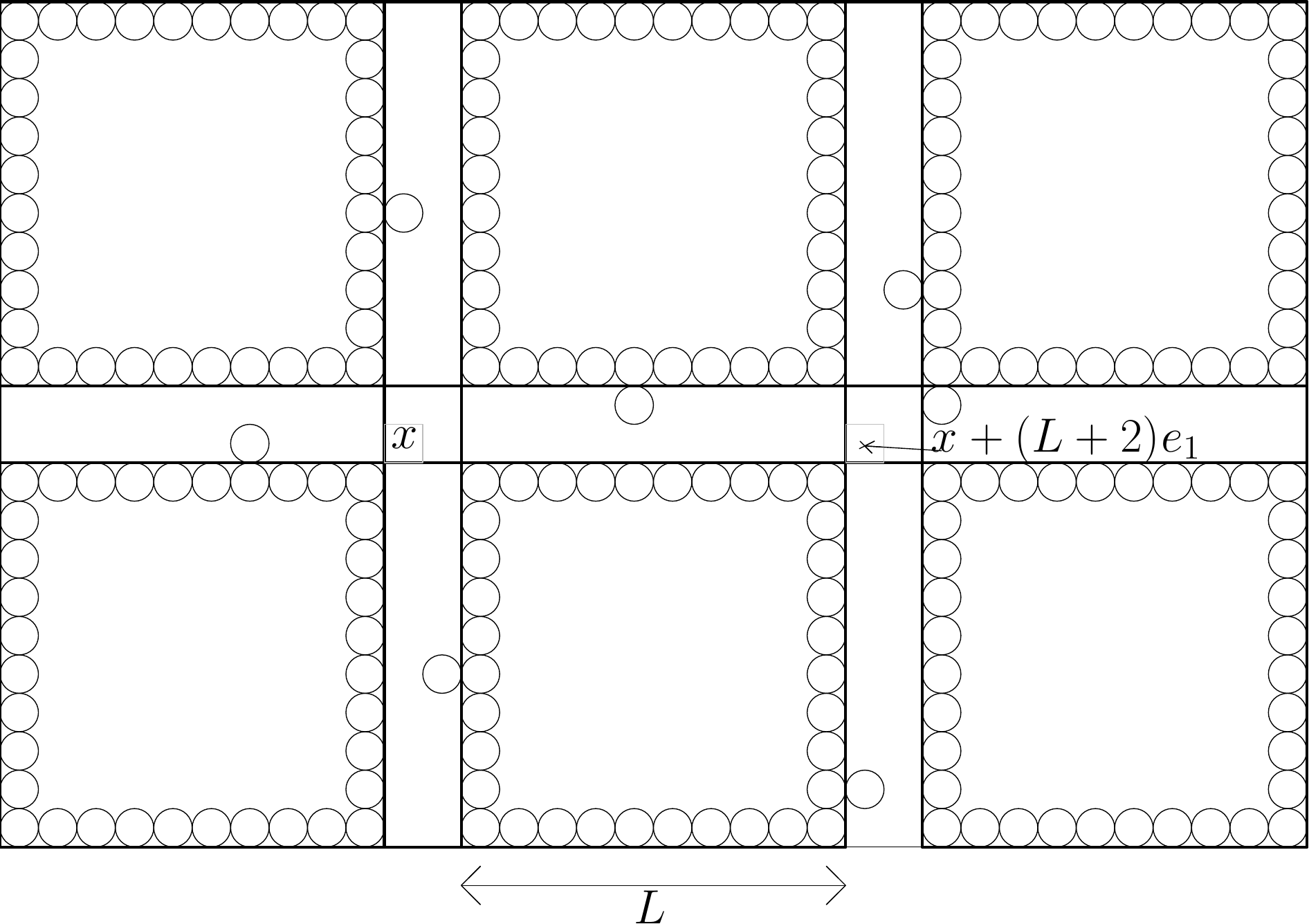}
\caption{A configuration in which the block neighborhood $\mathcal{N}_{x,1}$ is such that $\bar\eta_{x,x+(L+2)e_1}=1$ in dimension $2$. We represent the frameable blocks as already framed.}
\label{fig:openedge}
\end{center}
\end{figure}

\begin{lemma}\label{l:depperco}
$\bar\mu$ is a $(d+2)$--dependent bond percolation such that for any fixed $\rho\in (0,1)$, $\bar\mu(\bar\eta_{0,(L+2)e_i}=1)\underset{L\rightarrow\infty}{\longrightarrow} 1$ for all $i=1,\ldots,d$. In particular, for $L(\rho)$ large enough, there is an infinite open cluster.
\end{lemma}
\begin{proof}
To bound the dependence range, it is enough to check that for $x,y\in (L+2)\Z^d$ at distance at least $(L+2)(d+2)$, $\mathcal{N}_{x,i}$ and $\mathcal{N}_{y,j}$ are disjoint for any $i,j=1,\ldots,d$.

We now show that the percolation parameter goes to $1$ with $L$. First, the number of blocks in $\mathcal{N}_{x,i}$ depends only on $d$ and the configurations inside the different blocks in $\mathcal{N}_{x,i}$ are independent, so we just need to show that the probability for each block to satisfy condition \eqref{cond1} or \eqref{cond2} (depending on its $L$--dimension) goes to one. This is clearly true for condition \eqref{cond1}, since the probability that a given tube contains a zero is $1-\rho^{2^{d-1}L}$. For condition~\eqref{cond2}, consider a block of $L$--dimension $n$ with $n\geq 2$, and notice that under either $\mu$ or $\mu_0$, the configurations inside the $2^{d-n}$ different $n$-dimensional faces of the block are independent since the faces are disjoint. The conclusion therefore follows from Lemma~\ref{frameability}.
\end{proof}

Now we can define the auxiliary process $(Y_t)_{t\geq 0}$, which lives on $(L+2)\Z^d$ and whose diffusion coefficient we will compare with $D(\rho)$. Fix $L(\rho)$ so that under $\bar\mu$ the open cluster percolates. $Y$ is the simple random walk on the infinite percolation cluster. More precisely, let $\bar\mu^*:=\bar\mu(\cdot|0\leftrightarrow\infty)$, where as usual we write ``$0\leftrightarrow\infty$'' for ``$0$ belongs to the infinite percolation cluster''. $Y_0:=0$ and from $x\in (L+2)\Z^d$, $Y$ jumps to $x\pm(L+2)e_i$ at rate $\bar\eta_{x,x\pm(L+2)e_i}$. We write $\P^{\rm aux}_{\bar\mu^*}$ for the distribution of this random walk.

\begin{proposition}\label{prop:auxiliary}
For $L(\rho)$ large enough, there exists a positive (non-degenerate) $d\times d$ matrix $D_{\rm aux}(\rho)$ such that under $\P^{\rm aux}_{\bar\mu^*}$,
\begin{equation}\label{Bt}
\e Y_{\e^{-2}t}\underset{\e\rightarrow 0}{\longrightarrow}\sqrt{2D_{\rm aux}(\rho)}B_t,
\end{equation}
where $B$ is a standard $d$-dimensional Brownian motion and the convergence holds in the sense of
weak convergence of path measures on $D([0,\infty),\R^d)$. The matrix $D_{\rm aux}(\rho)$ is characterized by
\begin{multline}\label{varformulaaux}
u\cdot D_{\rm aux}(\rho)u=\inf_{f}\left\{\sum_{y\overset{(L+2)\Z^d}{\sim} 0}\bar\mu^*\left(\bar\eta_{0,y}[u\cdot y+f(\tau_{y}(\bar\eta))-f(\bar\eta)]^2\right)\right\}>0
\end{multline}
for any $u\in\R^d$, where the infimum is taken over local functions $f$ on $\lbrace 0,1\rbrace^{\mathcal{E}((L+2)\Z^d)}$.
\end{proposition}

\begin{proof}
The convergence to Brownian motion (formula \ref{Bt}) was proved in \cite{dMFGW} in the case of independent bond percolation and the variational formula (the equality in \eqref{varformulaaux} was established in \cite{spohn}). As pointed out  in Remark 4.16 of \cite{dMFGW}, independence is only needed to show positivity of the diffusion coefficient. This property indeed relies on the fact that the effective conductivity in a box of size $N$ is bounded away from $0$ as $N\rightarrow\infty$. Therefore to prove the positive lower bound of formula \ref{varformulaaux} we just need to show that this property holds under $\bar{\mu}$ if $L$ is large enough. Notice that we just need to prove the result in dimension $d=2$. In fact, in order to prove that $e_1\cdot D_{\rm aux}e_1>0$, we just need to find a lower bound on the number of disjoint open paths from left to right in $[1,(L+2)N]^2$ (\cite[Proposition 3.2]{chayes}). The other directions are similar.
More precisely, we only need to show that for $L$ large enough there exists $\lambda>0$ such that for $N$ large enough, 
\begin{equation}\label{disjointpaths}
\bar{\mu}(\text{at least $\lambda N$ disjoint left-right open paths in $[1,(L+2)N]^2$})\geq 1- e^{-\lambda N}.
\end{equation}
To this aim, we embed open paths in $\bar\mu$ into open paths of yet another percolation process built from $\mu$. Let $\hat\mu$ be the independent site percolation process defined as follows. The underlying graph is $\hat{\Z^2}=3(L+2)\Z\times 2(L+2)\Z$. For $x\in\Z^2$, we let $\hat{x}=(L+2)(3x_1,2x_2)$ and $\hat{\mathcal N}_x$ be the union of $\mathcal{N}_{x,1}$ and the tubes just above and to the right (Figure~\ref{fig:tildeN}). We say that $\hat{x}$ is $^\wedge$-open if each tube (resp.\@ each block) in $\hat{\mathcal N}_x$ satisfies condition~\eqref{cond1} (resp.\@ \eqref{cond2}). This defines a probability measure $\hat\mu$ on $\lbrace 0, 1\rbrace^{\hat{\Z^2}}$, which is an independent site percolation process since $\hat{\mathcal N}_x\cap\hat{\mathcal N}_{x'}=\emptyset$ if $x\neq x'$. Moreover, $\hat\mu(0\text{ is $^\wedge$-open})\underset{L\rightarrow\infty}{\longrightarrow}1$, similarly to what we proved in Lemma~\ref{l:depperco}. We now use \cite[Theorem 11.1]{kesten} to say that \eqref{disjointpaths} holds with $\bar\mu$ replaced with $\hat\mu$ and ``open'' by ``$^\wedge$-open''. In order to deduce \eqref{disjointpaths}, notice $\hat\mu$ is transparently coupled with $\bar\mu$ since they are both constructed from $\mu$. Moreover, it is clear that for $x\in\Z^2$ the following holds
\begin{itemize}
\item $\hat x$ is $^\wedge$-open implies $(\hat x,\hat x+(L+2)e_1)$ is open (for our dependent percolation process),
\item $\hat x,\hat x+3(L+2)e_1$ are $^\wedge$-open implies $(\hat x,\hat x+(L+2)e_1),(\hat x+(L+2)e_1,\hat x+2(L+2)e_1),(\hat x+2(L+2)e_1,\hat x+3(L+2)e_1)$ are open,
\item $\hat x,\hat x+2(L+2)e_2$ are $^\wedge$-open implies $(\hat x,\hat x+(L+2)e_2),(\hat x+(L+2)e_2,\hat x+2(L+2)e_2)$ are open.
\end{itemize}
Therefore, for the natural coupling between $\bar\mu$ and $\hat\mu$, existence of disjoint $^\wedge$-open paths implies existence of disjoint open paths and \eqref{disjointpaths} follows.

\begin{figure}
\begin{center}
\includegraphics[scale=.4]{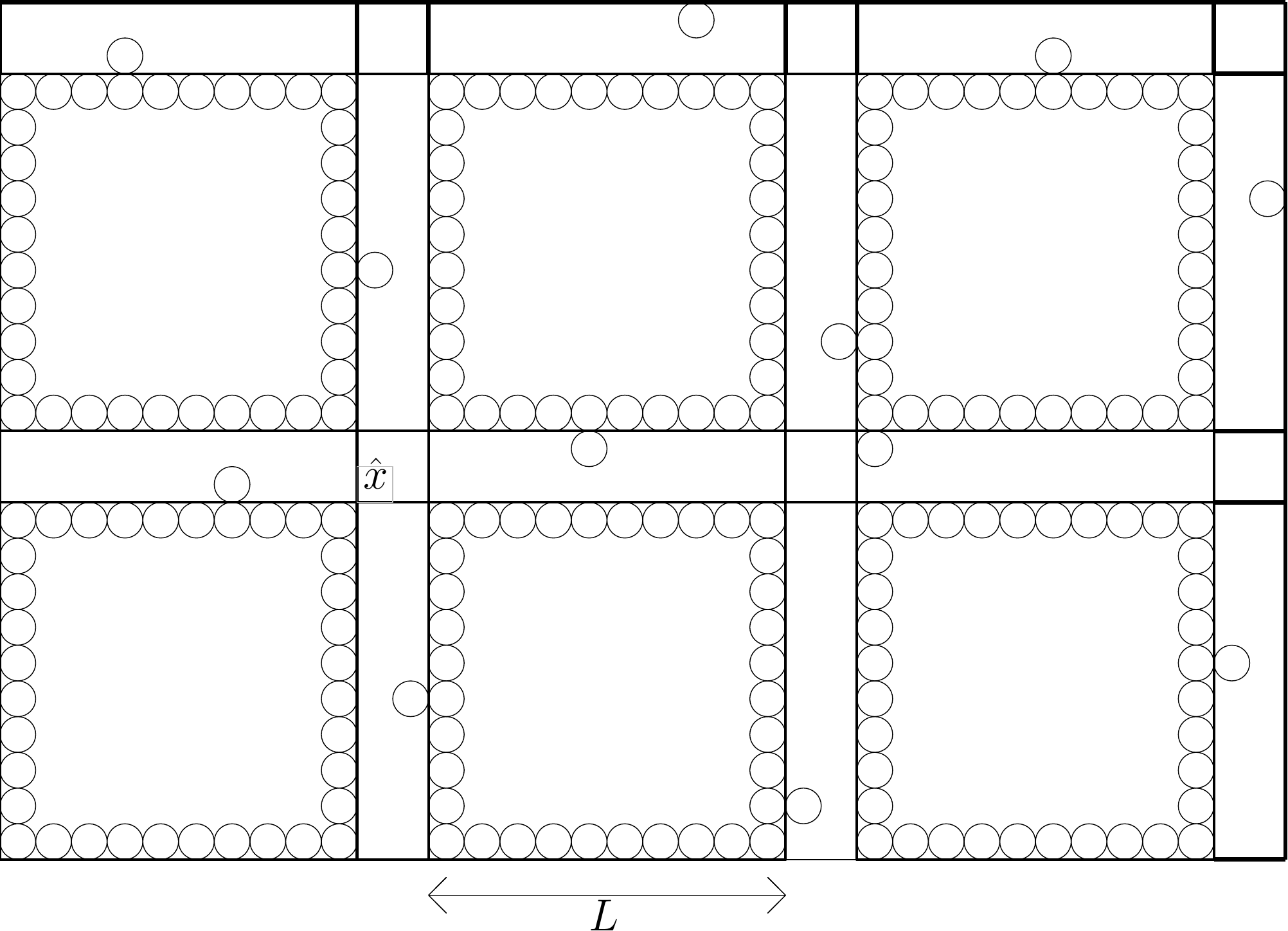}
\caption{$\hat{\mathcal N}_x$ in a configuration such that $\hat{x}$ is $^\wedge$-open.}
\label{fig:tildeN}
\end{center}
\end{figure} 

\end{proof}
\section{Comparison of the diffusion coefficients and proof of Theorem \ref{mainth}}
\label{comparison}
The main result of this section is the following Theorem, which states  that the self diffusion matrix for KA is lower bounded by the self diffusion matrix for the auxiliary model introduced in the previous section.

\begin{theorem}\label{Theo:compa}
There exists a constant $C=C(d,L(\rho))>0$ such that for all $i=1,\ldots d$,
\[
e_i\cdot D(\rho)e_i\geq Ce_i\cdot D_{\rm aux}(\rho)e_i.
\]
\end{theorem}

This result will be proved by using the variational characterisation of the diffusion matrices and via path arguments.
More precisely,  for any move $(x,\xi)\to (x',\xi')$ which has rate $>0$ for the auxiliary process we construct (in Lemmata~\ref{moves1}, \ref{moves2}, \ref{moves3}, \ref{moves4}, \ref{moves5}) a path of moves, 
each having positive rate for the KA process and connecting $(x,\xi)$ to $(x',\xi')$.
 Once Theorem \ref{Theo:compa} is proved, our main result follows.
\begin{proof}[Proof of Theorem \ref{mainth}]The result follows by using  Proposition \ref{prop:auxiliary} and Theorem \ref{Theo:compa}.\end{proof}

We are therefore left with the proof of Theorem \ref{Theo:compa}. Let us start by establishing some key Lemmata. \\
Let $A=\left\{ \xi(0)=1, \xi(x)=0\ \forall x\in \lbrace 0,1\rbrace^d \setminus 0\right\}$. Define $\mu_A=\mu(\cdot|A)$ and denote by $\eta^{x y,{\square}}$ the configuration obtained from $\eta$ by exchanging the contents of the boxes $x+\lbrace 0,1\rbrace^d$ and $y+\lbrace 0,1\rbrace^d$. Then the following holds

\begin{lemma}\label{lemmacomparison}

\begin{equation}
e_i\cdot D_{\rm aux}e_i\leq \bar{\mu}(0\leftrightarrow\infty)^{-1} \inf_{f}\left\{\sum_{y\overset{(L+2)\Z^d}{\sim} 0}\mu_A\left(\bar\eta_{0,y}[y_i+f(\tau_{y}(\eta^{0y,\square}))-f(\eta)]^2\right)\right\},
\end{equation}
where the infimum is taken over local functions $f$ on $\lbrace 0,1\rbrace^{\Z^d}$.
\end{lemma}

\begin{proof}
Let $f$ be a local function on $\lbrace 0,1\rbrace ^{\Z^d}$. We associate with it a local function $\bar f$ on $\lbrace 0,1\rbrace^{\mathcal{E}((L+2)\Z^d)}$, defined  by $\bar f(\bar \eta)=\mu_A(f|\bar\eta)$. Then, for $y\overset{(L+2)\Z^d}{\sim} 0$, since $\bar\eta$ does not depend on the configuration $\eta $ inside $\lbrace 0,1\rbrace^d$ and $y+\lbrace 0,1\rbrace^d$, we can bound
\begin{multline}
\mu_A\left(\bar\eta_{0,y}[y_i+f(\tau_{y}(\eta^{0y,\square}))-f(\eta)]^2\right)\\
=  \mu_A\left(\bar\eta_{0,y}\mu_A\left([y_i+f(\tau_{y}(\eta^{0y,\square}))-f(\eta)]^2|\bar\eta\right)\right)\nonumber\\
\geq  \bar{\mu}\left(\bar\eta_{0,y}[y_i+\bar f(\tau_{y}(\bar \eta))-\bar f(\bar\eta)]^2\right)\\
\geq  \bar{\mu}^*\left(\bar\eta_{0,y}[y_i+\bar f(\tau_{y}(\bar \eta))-\bar f(\bar\eta)]^2\right)\bar{\mu}(0\leftrightarrow\infty).
\end{multline}
Therefore the result follows by \eqref{varformulaaux}.
\end{proof}

The next sequence of Lemmata will show that for all $\eta,y$ such that $\bar\eta_{0,y}=1$ and $\eta\in A$, there exists an allowed path from  $\eta$ to $\tau_y\left(\eta^{0y,\square}\right)$ of finite length.
In order to avoid heavy notations, we will sometimes adopt an informal description of the allowed paths in the proofs.
For simplicity, we state the results in the case $y=(L+2)e_1$, but the process would be the same in any direction. In the following, $c(d)$ denotes a constant depending only on $d$ which may change from line to line.

\begin{lemma}\label{moves1}
Let $\eta\in\lbrace 0,1\rbrace ^{\Z^d}$ such that $\bar\eta_{0,(L+2)e_1}=1$. 
Choose a block  of $L$--dimension $n\in[2,d]$ inside $\mathcal{N}_{0,1}$, call it $\Lambda$. Then, 
using at most $c(d)2^{L^d}$ allowed moves, one can empty every site on its interior boundary  $\partial_-\Lambda$ (see Figure~\ref{fig:openedge}).
\end{lemma}

\begin{proof}
For the blocks of $L$--dimension $d$, this follows from the condition on $\bar\eta$ which implies frameability of these blocks. The number of necessary moves is bounded by the number of configurations inside one block times the number of involved blocks. Then we deal with the blocks of $L$--dimension $d-1,\ldots 2$ iteratively. Note that the frameability condition given by the definition of $\bar\eta$ is such that a block of $L$--dimension $k\in\lbrace 2,\ldots,d-1\rbrace$ is frameable (in the sense of the KA process in dimension $d$) as soon as the neighboring blocks of $L$--dimension $k+1$ are framed. Indeed, for $k<d$, every site $x$ in a block of $L$--dimension $k$ is adjacent to a point in the interior boundary of $d-k$ different blocks of dimension $k+1$ (which belong to $\mathcal{N}_{0,1}$ by construction). Therefore the path allowed by KA model with parameter $k$ in order to frame the configuration (which exists thanks to condition (2)), is also allowed by KA model with parameter $d$ (since the missing $d-k$ empty sites are found in the interior boundary of the neighboring framed block).
\end{proof}

After this step, the tubes in $\mathcal{N}_{0,1}$ are {\sl wrapped by zeros}, namely for any site $x$ inside a tube, any neighbor of $x$ that belongs to a facilitating block (i.e.\@ to a block of $L$--dimension $\geq 2$) is empty. 
Next we notice that inside a tube wrapped by zeros, the jump of a particle to a neighboring empty site is always allowed (since the wrapping guarantees an additional zero in the initial and in the final position of the particle).  More precisely
the following holds 
\begin{lemma}\label{moves2}
Fix $i\in[1,\dots,d]$ and choose any configuration $\xi$ such that $ B^{(1)}_i$ is wrapped by zeros. Fix $x\sim y$ with $x\in B^{(1)}_i$ and $y\in B^{(1)}_i$. Then $c_{x,y}(\eta)=1$.
Therefore if $\xi,\xi'$ are two configurations that are both empty on $\partial B^{(1)}_i$, coincide outside of $B^{(1)}_i$ and have the same number of zeros inside $B^{(1)}_i$, then there is an allowed path with length $c(d)L$ from  $\xi$ to $\xi'$. Moreover, if $x,x'\in B^{(1)}_i$, and $\xi,\xi'$ have the same positive number of zeros inside the tube and the tracer respectively at $x,x'$, it takes at most $c(d)L$ allowed moves inside the tube to change $\xi$ into $\xi'$ and take the tracer from $x$ to $x'$.
\end{lemma}

\begin{proof}
One just needs to notice that the wrapping ensures the satisfaction of the constraint for any such exchange.
\end{proof}

\begin{lemma}\label{moves3}
Fix a configuration such that: there is at least one zero in each tube inside $\mathcal{N}_{0,1}$; each  such  tube is wrapped by zeros; the tracer is at zero and the remaining sites of   $\lbrace 0,1\rbrace ^d$ are empty. Then, the tracer can be moved to any position in $B_1^{(1)}+2e_1$ , namely for any $y\in B_1^{(1)}+2e_1$ there is an allowed path from $(0,\eta)$ to $(y,\eta')$ for at least a configuration $\eta'$.
\end{lemma}

\begin{proof}
It is clear that the tracer can get to $e_1$ and we can bring a zero to $2e_1$ thanks to Lemma~\ref{moves2}. Then we can exchange the configuration in $e_1$ and $2e_1$, take the zero in $e_1+e_2$ inside the tube (namely exchange the configuration in $e_1+e_2$ and $2e_1+e_2$ thanks to the empty site in $2e_1+2e_2$ guaranteed by the wrapping), use Lemma~\ref{moves2} again to get the  to the desired position, and take the zero back to $e_1+e_2$ (if the desired position is $2e_1+e_2$ there is no need to take the zero in $e_1+e_2$ inside the tube).
\end{proof}

\begin{lemma}\label{moves4}
Fix a configuration such that all tubes adjacent to $B^{(0)}$ are wrapped by zeros and they all contain a zero except possibly $B_1^{(1)}+2e_1$. If they do not contain the tracer, we can exchange the configurations in the slices $\lbrace 1\rbrace\times\lbrace 0,1\rbrace^{d-1}$ and $\lbrace 2\rbrace\times\lbrace 0,1\rbrace^{d-1}$ in at most $c(d)L$ allowed moves.
\end{lemma}

\begin{proof}
For $x\in\lbrace 1\rbrace\times\lbrace 0,1\rbrace^{d-1}$, $x+e_1$ has $d-1$ empty neighbors in $\lbrace 2\rbrace\times\Z^{d-1}$ thanks to the wrapping. Moreover $x$ is adjacent to $d$ tubes, $d-1$ of which are not $B_1^{(1)}+2e_1$ and therefore contain a zero that can be brought to a site adjacent to $x$ using Lemma~\ref{moves2}. The constraint for the exchange is then satisfied if the configurations differ at $x, x+e_1$ (else the exchange is pointless).

\end{proof}

\begin{lemma}\label{moves5}
Fix a configuration such that: all tubes adjacent to $B^{(0)}$ are wrapped by zeros and they all contain a zero except possibly $B_1^{(1)}+2e_1$; either the slice $\lbrace 0\rbrace\times\lbrace 0,1\rbrace ^{d-1}$ or the slice $\lbrace 1\rbrace\times\lbrace 0,1\rbrace ^{d-1}$ are completely empty. Then we can exchange the configurations in $\lbrace 0\rbrace\times\lbrace 0,1\rbrace ^{d-1}$ and $\lbrace 1\rbrace\times\lbrace 0,1\rbrace ^{d-1}$ in at most $c(d)L$ allowed moves.
\end{lemma}
\begin{proof}
We describe the case $\lbrace 0\rbrace\times\lbrace 0,1\rbrace ^{d-1}$ empty. Order arbitrarily the zeros in positions $x\in\lbrace 0\rbrace\times\lbrace 0,1\rbrace ^{d-1}$ and move them one by one to $x+1$. When attempting to move the $i$--th zero, initially in position $x$, a certain number $N_i$ of its neighbors in slice $\lbrace 0\rbrace\times\lbrace 0,1\rbrace ^{d-1}$ have not been touched and are still empty. The other $d-1-N_i$ zeros are now in neighboring positions of $x+e_1$. Moreover, there are $d-1$ tubes adjacent to both $x$ and $x+e_1$. In $N_i$ of those, we take the zero to the position adjacent to $x+e_1$, and in the other $d-1-N_i$ to the position adjacent to $x$. Now the condition to exchange the variables at $x,x+e_1$ is satisfied.
\end{proof}
We are now ready to  prove the following key result
\begin{lemma}\label{ultimo}
There exists a constant $C=C(L,\rho)<\infty$ such that for any $f$ local function on $\lbrace 0,1\rbrace^{\Z^d}$, we have
\begin{multline}
\mu_A\left(\bar\eta_{0,y}[y_i+f(\tau_{y}(\eta^{0y,\square}))-f(\eta)]^2\right)\\
\leq C(L,\rho)\Biggl[\sum_{y\in\Z^d\setminus\lbrace 0\rbrace}\  \sum_{z\sim y} \mu_0\left(c_{yz}(\eta)[f(\eta^{xy})-f(\eta)]^2\right)\\
\qquad+\sum_{y\sim 0}\mu_0\left(c_{xy}(\eta)[y_i+f(\tau_{y}(\eta^{0y}))-f(\eta)]^2\right)\Biggr].
\end{multline}
\end{lemma}

 \begin{proof}[Proof]
Due to Lemmata~\ref{moves1}, \ref{moves2}, \ref{moves3}, \ref{moves4}, \ref{moves5}, we know that for all $\eta$ such that $\bar\eta_{0,y}=1$ and $\eta\in A$, there exists an allowed path from  $\eta$ to $\tau_y\left(\eta^{0y,\square}\right)$ of length upper bounded by $C'2^{L^d}$ for some finite constant $C'$. In Figure~\ref{fig:path}, we give the main steps in the construction of such a path. 
In particular, $\sum_{k=0}^{N-1}\mathbf{1}_{x^{(k)}=0}y^{(k)}_i=y_i$.

\begin{figure}
\begin{center}
\includegraphics[scale=.4]{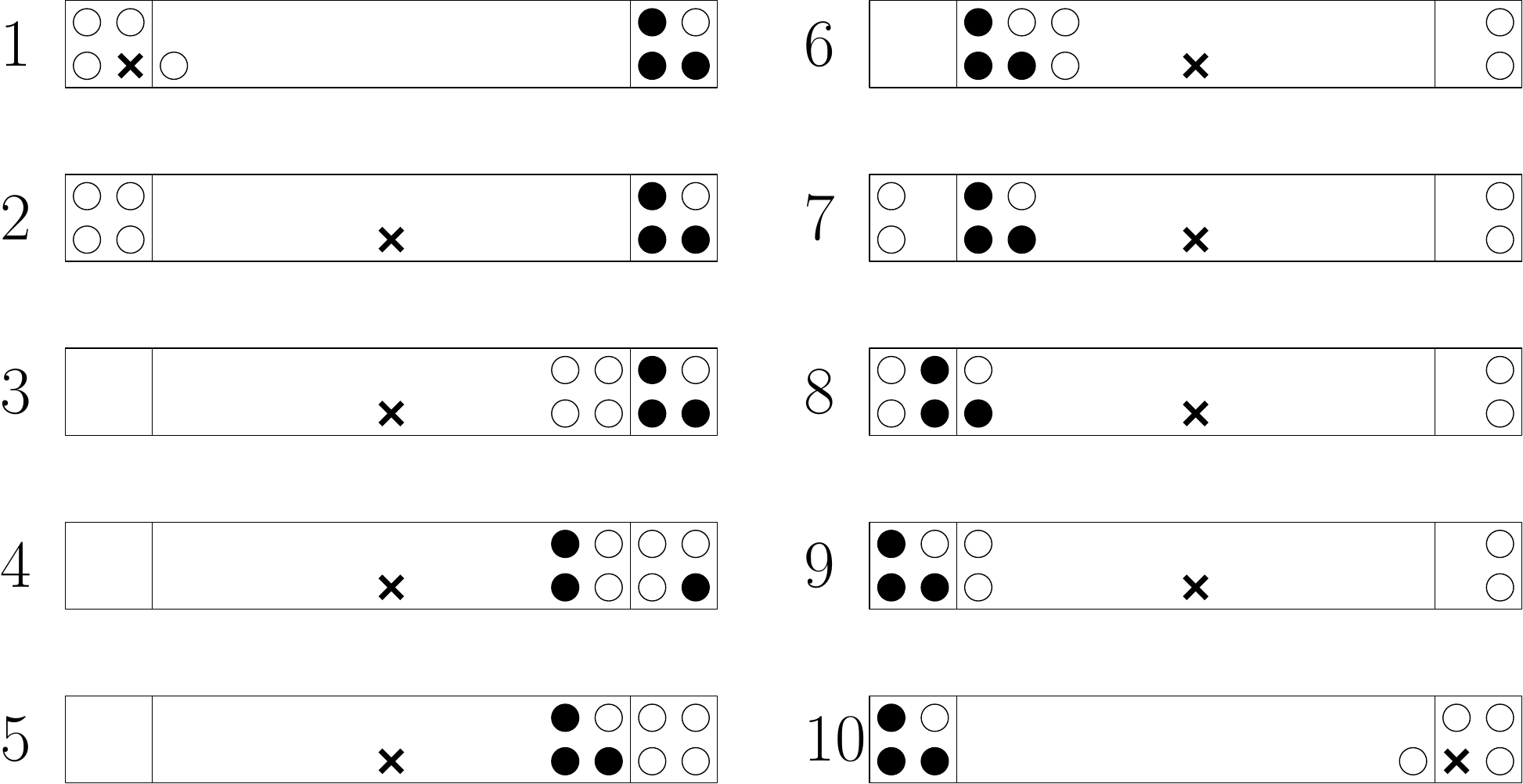}
\caption{An example of the main steps in the construction of the path from $\eta$ (in line 1) to $\tau_{y}(\eta^{0y,\square})$ (in line 10) when $y=(L+2)e_1$. Only the liaison tube is represented. From line 1 to 2 we use Lemma~\ref{moves3}; line 2 to 3: Lemmata~\ref{moves4} and \ref{moves2} twice; line 3 to 4: Lemmata~\ref{moves4} and \ref{moves2}; line 4 to 5: Lemmata~\ref{moves5}, \ref{moves4}; line 5 to 6: Lemmata~\ref{moves4}, \ref{moves2}; line 6 to 7: Lemmata~\ref{moves2}, \ref{moves4}, \ref{moves5}; line 7 to 8: Lemmata~\ref{moves4}, \ref{moves2}; line 8 to 9: Lemmata~\ref{moves5}, \ref{moves4}; from 9 to 10, Lemmata~\ref{moves2}, \ref{moves4}, \ref{moves3}.}
\label{fig:path}
\end{center}
\end{figure}

Then we can write 
\begin{eqnarray}
y_i+f(\tau_{y}(\eta^{0y,\square}))-f(\eta)=\sum_{k=0}^{N-1}\left[\mathbf{1}_{x^{(k)}=0}y^{(k)}_i+f(\eta^{(k+1)})-f(\eta^{k)})\right].
\end{eqnarray}
By Cauchy-Schwarz inequality, we deduce that
\begin{eqnarray}
[y_i+f(\tau_{y}(\eta^{0y,\square}))-f(\eta)]^2\leq C'2^{L^d}\sum_{k=0}^{N-1}c_{x^{(k)}y^{(k)}}(\eta^{(k)})\left[\mathbf{1}_{x^{(k)}=0}y^{(k)}_i+f(\eta^{(k+1)})-f(\eta^{k)})\right]^2.
\end{eqnarray}
Therefore,
\begin{multline}
\mu_A\left(\bar\eta_{0,y}[y_i+f(\tau_{y}(\eta^{0y,\square}))-f(\eta)]^2\right)= (1-\rho)^{1-2^d} \mu_0\left(\mathbf{1}_A\bar\eta_{0,y}[y_i+f(\tau_{y}(\eta^{0y,\square}))-f(\eta)]^2\right)\\
\leq (1-\rho)^{1-2^d}C'2^{L^d}\mu_0\left(\bar{\eta}_{0,y}\sum_{k=0}^{N-1}c_{x^{(k)}y^{(k)}}(\eta^{(k)})\left[\mathbf{1}_{x^{(k)}=0}y^{(k)}_i+f(\eta^{(k+1)})-f(\eta^{k)})\right]^2\right)\\
\leq(1-\rho)^{1-2^d}C'2^{L^d}\left\{\sum_{z\sim x\neq 0,\eta,\eta'}\mu_0(\eta)\bar{\eta}_{0,y}\sum_{k=0}^{N-1}\mathbf{1}_{x^{(k)}=x,y^{(k)}=z,\eta^{(k)}=\eta'}c_{xz}(\eta')\left[f(\eta'^{xz})-f(\eta')\right]^2,\right.\\
+\left.\sum_{z\sim 0,\eta,\eta'}\mu_0(\eta)\bar{\eta}_{0,y}\sum_{k=0}^{N-1}\mathbf{1}_{x^{(k)}=0,y^{(k)}=z,\eta^{(k)}=\eta'}c_{xz}(\eta')\left[z_i+f(\tau_z\eta'^{xz})-f(\eta')\right]^2\right\},
\end{multline}
where the sums are taken over $x\sim z$ inside $\mathcal{N}_{0,i}$, $\eta,\eta'\in\lbrace 0,1\rbrace^{\mathcal{N}_{0,i}}$ with the same number of zeros, and the equality $\eta^{(k)}=\eta'$ actually means $\eta^{(k)}=\tau_{Y_k}\eta'$, where $Y_k=\sum_{m=0}^{k-1}y^{(j)}$. Since $\eta$ and $\eta'$ have the same number of zeros and the tracer at zero by construction, $\mu_0(\eta)=\mu_0(\eta')$ and we can bound $\bar{\eta}_{0,y}\sum_{k=0}^{N-1}\mathbf{1}_{x^{(k)}=0,y^{(k)}=z,\eta^{(k)}=\eta'}$ by $N\leq C'2^{L^d}$ to obtain
\begin{multline}
\mu_A\left(\bar\eta_{0,y}[y_i+f(\tau_{y}(\eta^{0y,\square}))-f(\eta)]^2\right)\\
\leq (1-\rho)^{1-2^d}(C'2^{L^d})^3\Biggl[\sum_{x\neq 0}\  \sum_{z\sim x} \mu_0\left(c_{xz}(\eta)[f(\eta^{xz})-f(\eta)]^2\right)\\
\qquad+\sum_{z\sim 0}\mu_0\left(c_{0z}(\eta)[z_i+f(\tau_{z}(\eta^{0z}))-f(\eta)]^2\right)\Biggr].
\end{multline}
\end{proof}

Finally, we can conclude.
\begin{proof}[Proof of Theorem \ref{Theo:compa}]
The result follows from Lemma \ref{lemmacomparison}, Lemma \ref{ultimo} and the variational formula for $D$ in Proposition \ref{varD}.
\end{proof}


\begin{thebibliography}{}
\bibitem{bertini-toninelli}  Bertini, Lorenzo; Toninelli, Cristina {\em Exclusion processes with degenerate rates: convergence to equilibrium and tagged particle}. J. Statist. Phys. 117 (2004), no. 3-4, 549--580. 

\bibitem{orianediff} Blondel, Oriane {\em Tracer diffusion at low temperature in kinetically constrained models}. Ann. Appl. Probab. 25 (2015), no. 3, 1079--1107.

\bibitem{CMRT} Cancrini, N.; Martinelli, F.; Roberto, C.; Toninelli, C. {\em Kinetically constrained spin models}. Probab. Theory Related Fields 140 (2008), no. 3-4, 459--504.

\bibitem{CMRT2}  Cancrini, N.; Martinelli, F.; Roberto, C.; Toninelli, C. {\em Kinetically constrained lattice gases}. Comm. Math. Phys. 297 (2010), no. 2, 299--344.

\bibitem{chayes} Chayes, J. T.; Chayes, L. {\em Bulk transport properties and exponent inequalities for random resistor and flow networks}, Comm. Math. Phys. 105 (1986), no. 1, 133--152.

\bibitem{dMFGW}  De Masi, A.; Ferrari, P. A.; Goldstein, S.; Wick, W. D., {\em An invariance principle for reversible Markov processes. Applications to random motions in random environments.} J. Statist. Phys. 55 (1989), no. 3-4, 787--855.

\bibitem{parisi}
 {Franz, S.; Mulet, R.; Parisi, G.}, {\em Kob-Andersen model: A nonstandard mechanism for the glassy transition},
         {Phys.Rev.E}, 65,
     {021506}.

 \bibitem{GST}Garrahan, J.P.; Sollich, P.; Toninelli, C.,
       {\sl Kinetically constrained models},
        in "Dynamical heterogeneities in glasses, colloids, and granular
  media", Oxford Univ. Press, Eds.: L. Berthier, G. Biroli, J-P Bouchaud, L.
  Cipelletti and W. van Saarloos (2011).

\bibitem{GLT}
  Gonçalves, P.; Landim, C.; Toninelli, C. {\em Hydrodynamic limit for a particle system with degenerate rates}. Ann. Inst. Henri Poincaré Probab. Stat. 45 (2009), no. 4, 887--909.
    
\bibitem{kesten}  Kesten, Harry {\em Percolation theory for mathematicians}. Progress in Probability and Statistics, 2. Birkhäuser, Boston, Mass., 1982. iv+423 pp. ISBN: 3-7643-3107-0.

\bibitem{KA} {Kob, W.; Andersen, H. C.}, {\em Kinetic lattice-gas model of cage effects in high-density
    liquids and a test of mode-coupling theory of the ideal-glass transition},
    Physical Review E, 48, 4359--4363 (1993).
    
   
\bibitem{kurchan2}
    {Kurchan, J.; Peliti, L.; Sellitto, M.}, {\em Aging in lattice-gas models with constrained dynamics},
     {Europhys. Lett }, 39,
       {365--370} (1997).    
    
     Lee, Tzong-Yow; Yau, Horng-Tzer {\em Logarithmic Sobolev inequality for some models of random walks}. Ann. Probab. 26 (1998), no. 4, 1855--1873.
    
    
    
\bibitem{MP}
    Marinari, E.; Pitard, E., {\em Spatial correlations in the relaxation
of the Kob-Andersen model},
    {Europhysics Lett.},
   {69},
    {35-241}, (2005).    
    
\bibitem{nagahata}  Nagahata, Yukio {\em Lower bound estimate of the spectral gap for simple exclusion process with degenerate rates}. Electron. J. Probab. 17 (2012), no. 92, 19 pp.

\bibitem{quastel} Quastel, Jeremy {\em Diffusion of color in the simple exclusion process}. Comm. Pure Appl. Math. 45 (1992), no. 6, 623--679.

 \bibitem{Ritort} Ritort, F.; Sollich, P.,
{\em Glassy dynamics of kinetically constrained models},
Advances in Physics {52},219--342 (2003).

\bibitem{spohn}  Spohn, Herbert, {\em Tracer diffusion in lattice gases}. J. Statist. Phys. 59 (1990), no. 5-6, 1227--1239. 

  \bibitem{S}
Spohn H. :
\ {\em Large scale dynamics of interacting particles.}
\ Berlin: Springer 1991.

\bibitem{BiTo}
      Toninelli, Cristina; Biroli, Giulio {\em Dynamical arrest, tracer diffusion and kinetically constrained lattice gases}. J. Statist. Phys. 117 (2004), no. 1-2, 27--54.

\bibitem{TBF}  Toninelli, Cristina; Biroli, Giulio; Fisher, Daniel S. {\em Cooperative behavior of kinetically constrained lattice gas models of glassy dynamics}. J. Stat. Phys. 120 (2005), no. 1-2, 167--238.


\bibitem{yau} Yau, Horng-Tzer {\em Logarithmic Sobolev inequality for generalized simple exclusion processes}. Probab. Theory Related Fields 109 (1997), no. 4, 507--538.

 



 






\end{thebibliography}
\end{document}